\documentclass[10pt]{amsart}

\usepackage{amsmath,amsthm}
\usepackage{amsfonts}
\usepackage{amssymb}
\usepackage{enumerate}
\usepackage[usenames]{color}
\newcommand{\dx}{\operatorname{dx}}

\newcommand{\vol}{\operatorname{vol}}
\newcommand{\inj}{\operatorname{inj}}

\newcommand{\reals}{\mathbb{R}}

\newcommand{\II}{\operatorname{II}}
\newcommand{\graph}{\operatorname{graph}}

\newcommand{\Hom}{\operatorname{Hom}}
\newcommand{\Rm}{\operatorname{Rm}}
\newcommand{\proj}{\operatorname{proj}}
\newcommand{\Haus}{\mathcal{H}}
\newcommand{\ints}{\mathbb{Z}}

\newtheorem{theorem}{Theorem}[section]
\newtheorem{lemma}[theorem]{Lemma}
\newtheorem{corollary}[theorem]{Corollary}
\newtheorem{proposition}[theorem]{Proposition}
\theoremstyle{definition}
\newtheorem{definition}{Definition}[section]

\newtheorem{remark}{Remark}
\numberwithin{equation}{section}

\begin{document}

\title[Compactness for the Second Fundamental Form]{A Compactness Theorem for the Second Fundamental Form}

\author{Andrew A. Cooper}
\address{Department of Mathematics\\
Michigan State University\\
East Lansing, Michigan 48824}
\email{coope106@math.msu.edu}
\thanks{The author was partially supported by an RTG Research Training in Geometry and Topology NSF grant DMS 0353717 and as a graduate student on NSF grant DMS 0604759.}

\date{}

\begin{abstract}
	In this note we establish several versions of a compactness theorem for submanifolds.  In particular we require only bounds on the second fundamental form and do not assume volume or diameter bounds.  As an application we prove a compactness theorem for mean curvature flows and use it to construct smooth blow-up limits as singularity models.
\end{abstract}

\maketitle
	
\section{Introduction}
	A celebrated theorem of Cheeger and Gromov states that families of Riemannian manifolds with uniform $C^\ell$ bounds on the curvature tensor and a uniform lower bound on the injectivity radius are precompact in a certain sense:
	
	\begin{theorem}[Cheeger-Gromov \cite{gromov81}]\label{cheeger.gromov} Let $(N_k,h_k,x_k)$ be a sequence of complete pointed Riemannian manifolds such that $|\nabla^\ell\Rm (N_k,h_k)|\leq C$ for each $1\leq\ell\leq \ell_0$ and $\inj(h_k)\geq \eta>0$.  Then there is complete $C^{\ell_0+1}$ Riemannian manifold $(N_\infty, h_\infty, x_\infty)$ such that
		\begin{enumerate}
			\item $N_\infty$ admits a sequence of relatively compact open sets $V_1\subset V_2\subset\cdots\subset N_\infty$ which exhausts $N_\infty$ and embeddings $\psi_k:(V_k,x_\infty)\hookrightarrow (N_k,x_k)$, such that for each $R>0$ the $h_k$-metric ball $B(x_k,R)$ is contained in $\psi_k(V_k)$ for all $k\geq k_0(R)$
			\item $\psi_k^*h_k\rightarrow h_\infty$ in the $C^{\ell_0+1,\gamma}$ topology on compact subsets of $N_\infty$, for any $0\leq \gamma<1$
		\end{enumerate}
	\end{theorem}
	
	This theorem has been used extensively in the theory of singularities of the Ricci flow \cite{hamiltoncompactness} and to prove topological finiteness theorems \cite{cheegerthesis}. Our goal is to establish an analogous compactness theorem for Riemannian immersions.
	
	Given an immersion $F:M\looparrowright (N,h)$ of a compact $m$-manifold $M$, we may equip $M$ with a background Riemannian metric and isometrically embed $(N,h)$ into some Euclidean space $\reals^K$.  This allows us to consider the space $C^\ell(M,N)$ of $C^\ell$ maps from $M$ to $N$.  The curvature of the image submanifold $F(M)$ is invariant under reparametrization of $M$; thus bounds on the curvature of $F(M)$ do not allow us to appeal directly to the Arzela-Ascoli theorem for compactness of families of immersions $F:M\looparrowright N$. In fact by composing with a diffeomorphism of $M$, we may make any derivative of $F$ arbitrarily large without changing the extrinsic curvature.  The content of the our main theorem is that this diffeomorphism-invariance can be corrected for in a way that allows us to use Arezela-Ascoli, albeit at the cost of possible topological change.

	\begin{theorem}\label{geometric.convergence.ambient}
			Let $M_k^m$ be smooth closed $m$-manifolds and $(N_k,h_k)$ smooth Riemannian ${m+n}$-manifolds such that $|\nabla^\ell\Rm(N_k,h_k)|\leq C_\ell$, $0\leq\ell\leq\ell_0$, and $\inj(N_k,h_k)\geq \eta>0$. Suppose $F_k:(M_k,p_k)\looparrowright (N_k,h_k,x_k)$ are a sequence of pointed immersions of $M_k$ into $(N_k,h_k)$ such that the second fundamental forms and their covariant derivatives are bounded pointwise, i.e. $|\nabla^\ell\II_k|\leq C_\ell$, $0\leq\ell\leq\ell_0$.   Then there exist a $C^{\ell_0+1}$ $m$-manifold $(M_\infty,p_\infty)$ and a complete Reimannian manifold $(N_\infty,h_\infty,x_\infty)$ such that:
			\begin{enumerate}
				\item $M_\infty$ admits an exhausting sequence $W_1\subset W_2\subset\cdots$ of relatively compact open sets and embeddings $\phi_k:(W_k,p_\infty)\hookrightarrow (M_k,p_k)$, such that for any $R>0$, the $F_k^*h_k$-metric ball $B(p_k,R)$ is contained in $\phi_k(W_k)$ for all $k\geq k_0(R)$
				\item $N_\infty$ admits an exhausting sequence $V_1\subset V_2\subset\cdots$ of relatively compact open sets and embeddings $\psi_k:(V_k,x_\infty)\hookrightarrow (N_k,x_k)$, such that for any $R>0$, the $h_k$-metric ball $B(x_k,R)$ is contained in $\psi_k(V_k)$ for all $k\geq k_0(R)$
				\item $\psi_k^*h_k\rightarrow h_\infty$ on compact sets in the $C^{\ell_0+1,\gamma}$ topology for any $
					0\leq\gamma<1$
				\item $\phi_k(W_k)\subset \psi_k(V_k)$.
				\item $\psi_k^{-1}\circ F_k\circ\phi_k\rightarrow F_\infty$ on compact sets in the $C^{\ell_0+1,\gamma}$ topology for any $0\leq\gamma<1$
				\item $(M_\infty, F_\infty^*h_\infty)$ is a complete Riemannian manifold
			\end{enumerate}
		\end{theorem}
	
	Here the $C^{\ell_0+1,\gamma}$ topology is that given by isometrically embedding $N_\infty$ into some Euclidean space $\reals^K$ and equipping $M_\infty$ with a background metric.
	
	We refer to convergence as in the conclusion of Theorem \ref{geometric.convergence.ambient} as {\em convergence in $C^{\ell_0+1,\gamma}$ in the geometric sense}.  We note that in case $m=0$, $M_k=\{p_k\}$, our theorem recovers the Cheeger-Gromov theorem.
			
	\bigskip
	
	In section \ref{section.finiteness}, we apply Theorem \ref{geometric.convergence.ambient} to prove some finiteness theorems for classes of immersions $F:M\looparrowright N$.  In particular, we establish
		\begin{theorem}\label{regular.homotopy.finite.fixedtargetdomain}
			Let $\mathcal{C}=\{c_i\}$ be a collection of regular homotopy classes of maps $F:M\rightarrow (N,h)$, up to diffeomorphism of $M$.  If $\mathcal{C}$ is infinite, then there is no choice of immersed representatives $F_i\in c_i$ which satisfies $\vol(F_i(M))\leq C_1$, $|\II(F_i(M))|\leq C_2$.
		\end{theorem}
	
	\bigskip
		
	In section \ref{section.blowup}, by keeping track of the dependence of the maps $\psi_k,\phi_k$ and the neighbourhoods $W_k, V_k$, we prove a compactness theorem for mean curvature flows in the $C^{\ell,\alpha}$ geometric sense, and use it to construct smooth singularity models for finite-time singularities of the mean curvature flow. This project was discussed by Chen and He \cite{chen&he08}.
		
		\begin{theorem}\label{smoothblowup}
			Let $F:M^m\times[0,T)\rightarrow (N^{m+n},h)$ be a mean curvature flow of compact submanifolds in a Riemannian manifold with bounded geometry.  Suppose $T<\infty$ is the first singular time.  Then there exists a smooth mean curvature flow $F_\infty:M^m\times(-\infty,0)\rightarrow (\reals^{m+n},dx^2)$ which models the singularity of $F$ at time $T$.
		\end{theorem}
		
	The sense in which $F_\infty$ models the singularity will be discussed in detail in section \ref{section.blowup}.  We note  that singularities of the mean curvature flow have been studied extensively using the {\em tangent flow}, which is a non-smooth mean curvature flow.  We will discuss how our {\em smooth blow-up} relates to the tangent flow.

\section{The Construction}
The idea of the proof of the Theorem \ref{geometric.convergence.ambient} is due essentially to Langer \cite{langer85}.  We will go over the construction in detail for the case when the ambient manifold is Euclidean, and then indicate how the construction can be extended to an arbitrary Riemannian manifold with bounded geometry.
	
	\subsubsection{Euclidean Case}\label{section.euclidean.construction}
	
	We begin by considering the case of the graph of a map $f:\reals^m\rightarrow\reals^n$, as in \cite{reilly73}.  We need to compare the standard square-norm of certain objects, e.g. $\left|D^2f\right|^2=\sum_{\substack{1\leq \alpha\leq n\\1\leq i,j\leq m}}\left(\frac{\partial^2f_\alpha}{\partial x_i\partial x_j}\right)^2$, with the norms of the tensors $\II$ and $\nabla\II$ in the metric $g$ induced by the immersion. To keep the norms straight, in this section we use $|\cdot|$ for the standard square-norm and $|\cdot|_g$ for the norm in $g$:\begin{equation}\begin{aligned}
	\left|\II\right|_g^2=&h_{ij\alpha}h_{kl\beta}g^{\alpha\beta}g^{ik}g^{jl}\\
	\left|\nabla\II\right|^2_g=&\nabla_ih_{jk\alpha}\nabla_ph_{qr\beta}g^{ip}g^{jq}g^{kr}g^{\alpha\beta}\end{aligned}\end{equation}

	\begin{lemma}\label{hessianbound}Let $f:D_r^m\rightarrow \reals^n$ be a $C^2$ function on the disc of radius $r$.  Then\begin{align*}
		\left|D^2 f\right|^2\leq (1+|Df|^2)^3\left|\II\right|_g^2\end{align*}where $\II$ is the second fundamental form of the graph of $f$.\end{lemma}
	\begin{proof}
		The graph of $f$ has immersion map $F(x_1,\ldots,x_m)=(x_1,\ldots,x_m,f_1,\ldots,f_n)$. We use the following tangent and normal frames, where $1\leq i\leq m$ and $1\leq \alpha\leq n$:\begin{equation}\begin{aligned}
			e_i=&(0,\ldots,0,1,0,\ldots,0,\frac{\partial f_1}{\partial x_i},\ldots,\frac{\partial f_n}{\partial x_i})=(0,\ldots,0,1,0,\ldots,0,D_i f)\\
			\nu_\alpha=&(-\frac{\partial f_\alpha}{\partial x_1},\ldots,-\frac{\partial  f_\alpha}{\partial x_m},0,\ldots,0,1,0,\ldots,0)=(-D f_\alpha,0,\ldots,0,1,0,\ldots,0)\end{aligned}\end{equation}
			
		These choices induce the metric on the tangent bundle of the graph, which we denote by $g$ with Latin indices:
			\begin{equation}
			g_{ij}=e_i\cdot e_j=\delta_{ij}+D_i f\cdot D_j f\end{equation}
			
		We also get a metric on the normal bundle, which we denote by $g$ with Greek indices:\begin{equation}
			g_{\alpha\beta}=\nu_\alpha\cdot\nu_\beta=\delta_{\alpha\beta}+D f_\alpha\cdot D f_\beta\end{equation}
		We will use $g^{ij}$ to denote the inverse matrix to $g_{ij}$ and $g^{\alpha\beta}$ to denote the inverse to $g_{\alpha\beta}$.
		We compute the second fundamental form. Note that $D^2F=(0,D^2 f)$. So we have\begin{equation}\begin{aligned}
			\II(e_i,e_j)=&\proj^\perp(D^2F(e_i,e_j))\\
			=&(D_{ij}^2F\cdot\nu_\beta)g^{\alpha\beta}\nu_\alpha\\
			=&\frac{\partial^2 f_\beta}{\partial x_i\partial x_j}g^{\alpha\beta}\nu_\beta\end{aligned}\end{equation}
			
		In components, $h_{ij\alpha}=\frac{\partial^2 f_\alpha}{\partial x_i\partial x_j}$.
	
		Then the norm-squared of the second fundamental form is\begin{equation}\begin{aligned}
			\left|\II\right|_g^2=&\frac{\partial^2 f_\alpha}{\partial x_i\partial x_j}\frac{\partial^2  f_\beta}{\partial x_k\partial x_l}g^{\alpha\beta}g^{ik}g^{jl}\end{aligned}.\end{equation}
		We can think of $\left|\II\right|_g^2$ as the norm-squared of $D^2 f$ in the metric $g$ as opposed to the standard metric.  We will compare $g^{\alpha\beta}$ and $g^{ij}$ to the standard metric by giving estimates for the eigenvalues of $g^{\alpha\beta}$ and $g^{ij}$. To do this we estimate the eigenvalues of $g_{ij}$ and $g_{\alpha\beta}$.
	
		Since $g_{\alpha\beta}=\delta_{\alpha\beta}+Df_\alpha\cdot Df_\beta$, we have that each eigenvalue  $\lambda$ of $g_{\alpha\beta}$ has
			\begin{equation}
				1\leq \lambda\leq 1+|Df|^2
			\end{equation}
		and similiarly the eigenvalues $\mu$ of $g_{ij}$ are bounded by 
			\begin{equation}
				1\leq \mu\leq 1+|D f|^2.
			\end{equation}
	
		Thus the eigenvalues of the inverse matrices $g^{\alpha\beta}$ and $g^{ij}$ are bounded away from zero and infinity:\begin{equation}\label{inveigen}\begin{aligned}
			1 \geq \lambda^{-1}\geq& \frac{1}{1+|D f|^2}\\
			1 \geq \mu^{-1}\geq& \frac{1}{1+|D f|^2}\end{aligned}\end{equation}
			
		So we can estimate\begin{equation}\begin{aligned}
			\left|\II\right|_g^2=&\frac{\partial^2 f_\alpha}{\partial x_i\partial x_j}\frac{\partial^2  f_\beta}{\partial x_k\partial x_l}g^{\alpha\beta}g^{ik}g^{jl}\\
			\ \geq & \  \sum_{\substack{1\leq\alpha\leq n\\1\leq i,j\leq m}}\left(\frac{\partial^2 f_\alpha}{\partial x_i\partial x_j}\right)^2\frac{1}{(1+|D f|^2)(1+|D f|^2)^2}\\
			= & \left|D^2 f\right|^2\frac{1}{(1+|D f|^2)^3}\end{aligned}\end{equation}
		which establishes our lemma.
			
	\end{proof}
	
	We may similarly bound the higher derivatives of $ f$ in terms of $Df$ and the covariant derivatives of $\II$:
	
	\begin{lemma}\label{higherderivs}For any $\ell\geq 2$, we can bound $\left|D^\ell f\right|$ in terms of $|D f|$, $\left|D^2 f\right|$, . . ., $|D^{\ell-1} f|$, $\left|\nabla^{\ell-2}\II\right|_g$, and absolute constants depending on $m$, $n$, and $\ell$.  In particular, the $\ell=3$ case is\begin{equation*}
		\left|D^3 f\right|\leq (1+|D f|^2)^2\left|\nabla\II\right|_g+2\sqrt{2m+4\sqrt{mn}+n}\left|D^2 f\right|^2|D f|\end{equation*}\end{lemma}
	\begin{proof}
		We will do the $\ell=3$ computation explicitly; the others are similar but more tedious.
		
		As in the proof of Lemma \ref{hessianbound}, we start by estimating $\left|\nabla\II\right|_g$ below.  To do this, we need to compute the Christoffel symbols for the tangent and normal bundles.  First we compute the tangential Christoffel symbols.  We compute $\nabla_{e_i}e_j$, the projection to the tangent space of $D_{e_i}e_j$:
		\begin{equation}\begin{aligned}
			\nabla_{e_i}e_j=&\proj^T(D_{e_i}e_j)\\
			=&\proj^T(\frac{\partial}{\partial x_i}(0,\ldots, 0,1,0,\ldots,0,D_j f))\\
			=&\proj^T((0,D^2_{ij} f))\\
			=&g^{kl}((0,D^2_{ij} f)\cdot e_l) e_k\\
			=&g^{kl}(D^2_{ij} f\cdot D_l f) e_k\end{aligned}\end{equation}
		so $\Gamma_{ij}^k=g^{kl}(D^2_{ij} f\cdot D_l f)$.  Similarly to compute $\Gamma_{i\alpha}^\beta$:\begin{equation}\begin{aligned}
			\nabla_{e_i}\nu_\alpha = &\proj^\perp(D_{e_i}\nu_\alpha)\\
				=&\proj^\perp(\frac{\partial}{\partial x_i}(-D f_\alpha,0,\ldots,0,1,0,\ldots,0))\\
				=&\proj^\perp(-\frac{\partial^2 f_\alpha}{\partial x_i\partial x_1},\ldots,-\frac{\partial^2 f_\alpha}{\partial x_i\partial x_m},0)\\
				=&g^{\beta\gamma}((-\frac{\partial^2 f_\alpha}{\partial x_i\partial x_1},\ldots,-\frac{\partial^2 f_\alpha}{\partial x_i\partial x_m},0)\cdot \nu_\gamma)\nu_\beta\\
				=&g^{\beta\gamma}(\sum_r\frac{\partial^2 f_\alpha}{\partial x_i\partial x_r}\frac{\partial f_\gamma}{\partial x_r})\\
				=&g^{\beta\gamma}(D^2_{i\cdot} f_\alpha\cdot D f_\gamma)\nu_\beta\end{aligned}\end{equation}
		so $\Gamma_{i\alpha}^\beta=g^{\beta\gamma}(D^2_{i\cdot} f_\alpha\cdot D f_\gamma)$.
	
		Then $\left|\nabla\II\right|_g^2$ is given by\begin{equation}\label{nablaII}\begin{aligned}
			\left|\nabla\II\right|_g^2=&(\frac{\partial h_{jk\alpha}}{\partial x_i} +h_{lk\alpha}\Gamma_{ij}^l + h_{jl\alpha}\Gamma_{ik}^l + h_{jk\beta}\Gamma_{i\alpha}^\beta)\\
			&\cdot(\frac{\partial h_{qr\gamma}}{\partial x_p} +h_{lr\gamma}\Gamma_{pq}^l + h_{ql\gamma}\Gamma_{pr}^l + h_{qr\delta}\Gamma_{p\gamma}^\delta)g^{ip}g^{jq}g^{kr}g^{\alpha\gamma}\\
			\geq&\frac{1}{(1+|D f|^2)^4}(\frac{\partial}{\partial x_i} h_{jk\alpha}+h_{lk\alpha}\Gamma_{ij}^l + h_{jl\alpha}\Gamma_{ik}^l + h_{jk\beta}\Gamma_{i\alpha}^\beta)^2\end{aligned}\end{equation}
		By (\ref{nablaII}) and Cauchy-Schwarz, we have\begin{equation}
			(1+|D f|^2)^4\left|\nabla\II\right|_g^2+2\left|D^3 f\right||B|\geq \left|D^3 f\right|^2\end{equation}
		where $B_{ijk\alpha}=h_{lk\alpha}\Gamma_{ij}^l+h_{jl\alpha}\Gamma_{ik}^l + h_{jk\beta}\Gamma_{i\alpha}^\beta$. It will suffice to bound $|B|$ above. Our estimates (\ref{inveigen}) for the eigenvalues of $g^{ij}$ and $g^{\alpha\beta}$ imply $\left|g^{ij}\right|^2\leq m$ and $\left|g^{\alpha\beta}\right|^2\leq n$.
			\begin{equation}\begin{aligned}
			|B|^2=&\sum_{\substack{1\leq\alpha\leq n\\1\leq i,j,k\leq m}}(h_{lk\alpha}\Gamma_{ij}^l+h_{jl\alpha}\Gamma_{ik}^l + h_{jk\beta}\Gamma_{i\alpha}^\beta)^2\\
			=&\sum_{\substack{1\leq\alpha\leq n\\1\leq i,j,k\leq m}}(\frac{\partial f_\alpha}{\partial x_l\partial x_k}g^{ls}(D^2_{ij} f\cdot D_s f)+\frac{\partial  f_\alpha}{\partial x_j\partial x_l}g^{ls}(D^2_{ik} f\cdot D_s f)\\
			&+\frac{\partial  f_\beta}{\partial x_j\partial x_k}g^{\beta\gamma}(D^2_{i\cdot} f_\alpha\cdot D f_\gamma))^2\\
			\ \leq& \  2\left|D^2 f\right|^2\left|g^{ij}\right|^2\left|D^2 f\right|^2|D f|^2\\
			&\ +4\left|D^2 f\right|^2\left|g^{ij}\right|\left|g^{\alpha\beta}\right|^2|D f|^2+\left|D^2 f\right|^2\left|g^{\alpha\beta}\right|^2\left|D^2 f\right|^2|D f|^2\\
			\ \leq& \  \left|D^2 f\right|^4|D f|^2(2m+4\sqrt{mn}+n)
			\end{aligned}\end{equation}
			
		Thus we have\begin{equation}
			\left|D^3 f\right|^2\leq (1+|D f|^2)^4\left|\nabla\II\right|_g^2+2\sqrt{2m+4\sqrt{mn}+n}\left|D^3 f\right|\left|D^2 f\right|^2|D f|\end{equation}
			
		The claimed estimate for $\left|D^3 f\right|$ follows from this and the quadratic formula.\end{proof}
		
	Next we want to realize any immmersion $F:M\rightarrow\reals^{m+n}$ as a collection of graphs over discs.
	
	We introduce the following notation and notions, following \cite{langer85}.  Given $q\in M$, denote by $A_q$ any Euclidean isometry which takes $F(q)$ to the origin and $T_{F(q)}F(M)$ to the plane $\{(x_1,\ldots,x_m,0)\}$.  Let $\pi$ be the projection of $\reals^{n+m}$ to the plane $\{(x_1,\ldots,x_m,0)\}$.  Define the {\em Langer chart at $q$} $U_{r,q}\subset M$ to be the component of $(\pi\circ A_q\circ F)^{-1}(D_r)$ which contains $q$.  
	
	We call $F:M\rightarrow \reals^{n+m}$ a $(r,\alpha)$-immersion if for each $q\in M$ there is some $ f_q:D_r^m\rightarrow\reals^n$ with $Df_q(0)=0$ and $|D f_q|\leq \alpha$ so that $A_q\circ F(U_{r,q})=\graph( f_q)$.  
	
	\begin{lemma}\label{arralpha} Let $\alpha>0$. Then for any $C^2$-immersed submanifold $F:M^m\rightarrow \reals^{n+m}$ and any $r$ satisfying
		\begin{align*}
			r\leq \frac{\alpha}{\left(1+\alpha^2\right)^{3/2}}\frac{1}{\sup_M\left|\II\right|_g}
		\end{align*}
		$F$ is a $(r,\alpha)$-immersion.\end{lemma}
	\begin{proof}
		Let $q\in M$ be arbitrary. Every submanifold is locally a graph over its tangent plane; thus $A_q(F(U_{r,q}))$ can be written as a graph over $D_r$ for small enough $r$.  So we set $S_q=\sup\{r|F(U_{r,q})=\graph( f_{r,q})\}$.  For any large $K$, if $F(U_{r,q})=\text{graph}( f_{r,q})$ and $|D f_{r,q}|\leq \frac{K}{2}$, we can extend $ f_{r,q}$ to have a larger domain and still $|D f|\leq K$.  Thus we have\begin{equation}\begin{aligned}
			\lim_{r\rightarrow S_q}\inf_{ f}\sup_{D_r}|D f|=\infty\end{aligned}\end{equation}
		where the infimum is taken over all $f$ with $Df(0)=0$ of which $A_qF(U_{r,q})$ is a graph.
		Thus for our given $\alpha$ there exists some $r_q,  f_{q}:D_{r_q}\rightarrow \reals^n$ with $\sup_{D_{r_q}}|D f_{q}|=\alpha$.  Now we use the fundamental theorem of calculus and Lemma \ref{hessianbound} to get\begin{equation}\label{fundamental.theorem}\begin{aligned}
			\alpha=\sup_{D_{r_q}}|D f_{q}|\leq r\sup_{D_{r_q}}\left|D^2 f_{q}\right|\leq r_q \left(1+\alpha^2\right)^{3/2}\sup_{D_{r_q}}\left|\II_{ f_q}\right|_g\end{aligned}\end{equation}
		which implies that \begin{equation}\label{Sq}
			r_q\geq \frac{\alpha}{\left(1+\alpha^2\right)^{3/2}}\frac{1}{\sup_{D_{r_q}}\left|\II_{ f_q}\right|_g}\geq\frac{\alpha}{\left(1+\alpha^2\right)^{3/2}}\frac{1}{\sup_M\left|\II\right|_g}.\end{equation}
		So for $r$ less than the right-hand side of (\ref{Sq}), there is some $ f:D_r\rightarrow\reals^{m+n}$ of which $A_q F(U_{r,q})$ is the graph, with $Df(0)=0$ and $|Df|\leq \alpha$.
		\end{proof}

	We will also make use of the following lemma, which relates the Langer atlas to the metric structure of $(M,F^*dx^2)$.

		\begin{lemma}\label{balls}
			Let $\alpha\leq \sqrt{3}$. In a $(r,\alpha)$-immersion $F:M^m\rightarrow\reals^{m+n}$, for any $0<\rho\leq \frac{r}{2}$, $\ell\in \mathbb{N}$, $q_0\in M$, we have that any metric ball $B(q_0,\ell\frac{\rho}{2})\subset M$ can be covered using at most $K^\ell$ Langer charts $U_{p,\frac{\rho}{4}}$ of radius $\frac{\rho}{4}$, such that $p\in B(q_0,\ell\rho)$, where $K=K(m,\alpha)$ is a constant depending only on the dimension $m$ and the constant $\alpha$.
			Moreover, we can assume that if $\ell_1\leq \ell_2$, the covering of $B(q_0,\ell_2\frac{\rho}{2})$ in $K^{\ell_2}$ Langer charts of radius $\frac{\rho}{4}$ contains the $K^{\ell_1}$ Langer charts used to cover $B(q_0,\ell_1\frac{\rho}{2})$.
		\end{lemma}
		\begin{proof}
			We proceed by induction.  If $\ell=2$, we have $B(q_0,\rho)\subset \subset U_{q_0,\rho}\subset U_{q_0,2\rho}$.  Since $\rho\leq \frac{r}{2}$, the projection of $F(U_{q_0,2\rho})$ to $T_{F(q_0)}F(M)$ is the $m$-ball of radius $2\rho$.  Let $K$ be the number of $m$-balls of radius $\frac{\rho}{4\sqrt{1+\alpha^2}}$ which can cover $D^m_{2\rho}$.  Each ball of radius $\frac{\rho}{4\sqrt{1+\alpha^2}}$ is contained in the projection of some $U_{p,\frac{\rho}{4}}$, $p\in U_{q_0,2\rho}$.  Thus we have that $B(q_0,\rho)$ can be covered by $K$ Langer charts of radius $\frac{\rho}{4}$. It is clear that the centers of these Langer charts can be taken to lie in $B(q_0,\rho)$.
			
			Now suppose $B(q_0,\ell\frac{\rho}{2})\subset \bigcup_{i=1}^{K^\ell}U_{p_i,\frac{\rho}{4}}$.  Then $B(q_0,(\ell+1)\frac{\rho}{2})$ is contained in a $\frac{\rho}{2}$ neighborhood of $\bigcup_{i=1}^{K^\ell}U_{p_i,\frac{\rho}{4}}$.  On the other hand, the $\frac{\rho}{2}$ neighborhood of each $U_{p_i,\frac{\rho}{4}}$ is contained in $B(p_i,\sqrt{1+\alpha^2}\frac{\rho}{4}+\frac{\rho}{2})$. Since $\alpha\leq \sqrt{3}$, we have
			\begin{equation}
				B(q_0,(\ell+1)\frac{\rho}{2})\subset \bigcup_{i=1}^{K^\ell}B(p_i,\rho)
			\end{equation}
		Each term in the union on the right-hand side can, by the definition of $K$, be covered by $K$ Langer charts of radius $\frac{\rho}{4}$, centered within distance $\rho$ of one of the $p_i$.  This completes the inductive step.\end{proof}

	\subsubsection{General Case}\label{section.ambient.construction}
	The general case of Lemma \ref{hessianbound} is:
		\begin{lemma}\label{hessianbound.ambient}
			Let $(N^{m+n},h)$ a Riemannian manifold, $x$ a point in $N$, $y:\reals^{m+n}\rightarrow (U,x)$ a coordinate chart of $N$. If $ f:D_r^m\rightarrow\reals^{n}$ is a $C^2$ function, then there exists $C$ depending on $|Dy|$ and $|D(y^{-1})|$ so that \begin{align*}
				\left|D^2 f\right|^2\leq C(1+|D f|^2)^3\left|\II\right|^2\end{align*}
			where $\II$ is the second fundamental form of $y(\graph f)$.
		\end{lemma}
		\begin{proof}
			Let $\overline{\II}$ be the second fundamental form of $\graph f$ considered as a submanifold of $(\reals^{m+n},y^*h)$.  Then $\left|\II\right|=\left|\overline{\II}\right|$, so we will compute $\left|\overline{\II}\right|^2$ as in the proof of Lemma \ref{hessianbound}.  We will abuse notation and write $h$ for $y^*h$. We will write $\overline{g}_{ij}$ for the metric induced on $\graph f$ by $h$, and $\overline{g}_{\alpha\beta}$ for the metric on the normal bundle of $\graph f$ with respect to $h$. We let $\theta$ be the least eigenvalue of $h$ and $\Theta$ the greatest eigenvalue of $h$.
			
			The tangent bundle of $\graph f$ is spanned, as before, by \begin{equation*}
				e_i=(0,\ldots,0,1,0,\ldots,0,D_i f)\end{equation*}
			So the induced metric is\begin{equation}\begin{aligned}
				\overline{g}_{ij}=&h(e_i,e_j)\\
					=&h_{ij}+ h(D_i f,D_j f)+h_{pj}\frac{\partial  f_p}{\partial x_i}+ h_{iq}\frac{\partial  f_q}{\partial x_j}
				\end{aligned}\end{equation}
			and the eigenvalues $\overline{\mu}=\overline{g}(X,X)$ of $\overline{g}_{ij}$ are therefore bounded by\begin{equation}\label{ambient.eigenvalue1}\begin{aligned}
				\theta \leq h(X,X)\leq \overline{\mu}\ \leq& \  h(X,X)\left(1+|D f|_{h}+|D f|^2_{h}\right)\\
				\ \leq& \  \Theta\left(1+\Theta^{\frac{1}{2}}|D f|+\Theta |D f|^2\right)\end{aligned}\end{equation}
			where $|\cdot|_{h}$ denotes the norm induced by $h$.
			
			The normal bundle $N_h$ is characterised by $N_h=\{X|h(X,e_i)\}=0$.  Equivalently $N_h=h^{-1}N_{dx^2}$, where $N_{dx^2}$ is the normal bundle of the graph with respect to the standard metric $dx^2$ and we consider $h$ as a bundle map over the identity $h:T\reals^{m+n}\rightarrow T\reals^{m+n}$.  We may thus take a normal frame $\overline{\nu}_\alpha=h^{-1}(\nu_\alpha)$.  Then to compute the eigenvalues of the normal metric $\overline{g}_{\alpha\beta}$, we consider $X\in \reals^n$ with $|X|^2=1$:
				\begin{equation}\label{ambient.eigenvalue2}\begin{aligned}
					\overline{g}(X,X)=\overline{g}_{\alpha\beta}X^\alpha X^\beta =&h(\overline{\nu}_\alpha,\overline{\nu}_\beta)X^\alpha X^\beta\\
						=&h(h^{-1}(\nu_\alpha),h^{-1}(\nu_\beta))X^\alpha X^\beta\\
						=&\nu_\alpha\cdot h^{-1}(\nu_\beta)X^\alpha X^\beta\\
						=&(-D(X\cdot f),X)\cdot y^{-1}(-D(X\cdot f),X)\\
						=&h^{-1}((-D(X\cdot f),X),(-D(X\cdot f),X))\end{aligned}\end{equation}
			Thus we have\begin{equation}
					\frac{1}{\Theta} \leq \overline{g}(X,X)\leq \frac{1}{\theta} (1+|D f|^2)\end{equation}
					
			Now just as in the proof of Lemma \ref{hessianbound}, we use (\ref{ambient.eigenvalue1}) and (\ref{ambient.eigenvalue2}) to bound the Hessian of $ f$ in terms of $\left|\overline{\II}\right|$, $|D f|$, and $\theta,\Theta$.  The eigenvalues of $h^{-1}$ are clearly controlled by $|Dy|$ and $|Dy^{-1}|$.					
		\end{proof}
	Similarly one can extend Lemma \ref{higherderivs} to a general ambient manifold:
		\begin{lemma}\label{higherderivs.ambient}
			Let $(N,h)$, $x$, and $y$ be as above. If $ f:D_r^m\rightarrow\reals^{n}$ is a $C^\ell$ function, then we can bound $\left|D^\ell f\right|$ in terms of $|D f|,\ldots, \left|D^{\ell-1} f\right|$, $\left|\nabla^{\ell-2}\II\right|$, $|Dy|,\ldots,\left|D^{\ell-1}y\right|$, and $\left|D(y^{-1})\right|,\ldots,\left|D^{\ell-1}(y^{-1})\right|$, where $\II$ and $\nabla$ are the second fundamental form and covariant derivative on $y(\graph f)$.
		\end{lemma}
	
	The proof of Cheeger-Gromov's theorem involves the following proposition, which is analogous to our Lemma \ref{arralpha}.  An exposition can be found in chapter 10, section 3 of \cite{petersenriemannian}. 
		\begin{proposition}\label{CG.charts}
			Suppose $(N^{m+n},h)$ is a Riemannian manifold with $\inj(N,h)\geq \eta>0$ and $\left|\nabla^\ell\Rm(N,h)\right|\leq C$ for $1\leq\ell\leq\ell_0$.  Then there exist $r_0$ and $Q$ depending on $C,\eta,\ell_0,m,n$  such that for any $0<r\leq r_0$, each $x\in N$ admits a chart $y_x:(U_x,x)\rightarrow (\reals^{m+n},0)$ such that $y_x(U_x)$ contains $D_r^{m+n}\subset\reals^{m+n}$ and such that $\left|Dy_x\right|,\ldots,\left|D^{\ell_0+2}y_x\right|$ and $\left|D(y^{-1}_x)\right|,\ldots,\left|D^{\ell_0+2}(y^{-1}_x)\right|$, and the derivatives of the transition maps are all bounded by $Q$
			
			Moreover we may take a subatlas with the property that the centers of the charts are some definite $0<\delta\leq \frac{r_0}{4}$ apart.
		\end{proposition}
		
	We refer to such an atlas as the Cheeger-Gromov atlas. 
		
	We are now ready to prove the version of Lemma \ref{arralpha} for a general ambient manifold. Toward this end, given $q\in M$, $r>0$, let the Langer chart $U_{r,q}$ be the component of $F^{-1}(y_{F(q)}(\pi^{-1}(D_r^m)))$ which contains $q$, where $\pi:\reals^{m+n}\rightarrow\reals^m$ is projection to the first $m$ coordinates.
	
		\begin{lemma}\label{arralpha.ambient}
			Let $F:M^m\rightarrow (N^{m+n},h)$ be an immersion such that $\left|\II\right|\leq C_1$, $\inj(N,h)\geq \eta>0$, and $\left|\Rm(N,h)\right|\leq C_2$.  For any $\alpha>0$, there is $r_1>0$ depending on $C_1,C_2,m,n,\alpha,\eta$ such that for any $0<r\leq r_1$, $q\in M$, $y_{F(q)}^{-1}(F(U_{r,q}))=\graph f$ for some $ f:D_r^m\rightarrow\reals^n$ with $|D f|\leq\alpha$.
		\end{lemma}
		\begin{proof}
			Let $\alpha>0, q\in M$ be arbitrary.  For small $r$, $y_{F(q)}^{-1}(F(U_{r,q}))$ is a graph over $D_r^m$.  So let $S_q$ be the supremum of such $r$.  If $S_q < r_0$, then the argument in the proof of Lemma \ref{arralpha} gives a lower bound on $S_q$ depending only on $\alpha,C_1,C_2$, and in particular independent of $q$.
			
			If $S_q=r_0$, we can write $y_{F(q)}^{-1}(F(U_{r_0,q}))$ as a graph of some $ f:D_{r_0}^m\rightarrow\reals^n$.  If $|D f|\leq \alpha$, we are done.  If $\sup_{D_{r_0}^m}|D f|> \alpha$, there is some smaller disc $D_r^m$ with $\sup_{D_r^m}|D f|=\alpha$; then the argument in the proof of Lemma \ref{arralpha} gives a lower bound on $S_q$ depending only on $\alpha, C_1,C_2$.
		\end{proof}

\section{Proof of Theorems \ref{geometric.convergence.euclidean} and \ref{geometric.convergence.ambient}}
\subsection{Immersions into $\reals^{m+n}$}
	We now prove the following theorem, which is the special case of Theorem \ref{geometric.convergence.ambient} when $(N_k,h_k)\equiv (\reals^{m+n},dx^2)$.
	
	\begin{theorem}\label{geometric.convergence.euclidean}
			Let $M_k^m$ be a sequence of smooth $m$-manifolds.  Suppose $F_k:(M_k,p_k)\rightarrow(\reals^{m+n},0)$ is an immersion of $M_k$ into $\reals^{m+n}$ such that the second fundamental forms $\II_k$ and their covariant derivatives $\nabla^\ell\II_k$ are bounded pointwise, $1\leq\ell\leq\ell_0$.  Then there exists a smooth $m$-manifold $(M_\infty,p_\infty)$ which admits a sequence of relatively compact open subsets $W_1\subset W_2\subset \cdots$ which exhausts $M_\infty$ and embeddings $\phi_k:(W_k,p_\infty)\hookrightarrow (M_k,p_k)$ such that:
			\begin{enumerate}
				\item $F_k\circ\phi_k$ subconverge in the $C^{\ell_0+1,\gamma}$ topology for any $
					0\leq\gamma<1$ on compact subsets of $M_\infty$ to some $F_\infty:(M_\infty,p_\infty)\rightarrow(\reals^{m+n},0)$.
				\item For each $R>0$, the metric ball $B_k(p_k,R)\subset (M_k,F_k^*dx^2)$ is contained in $\phi_k(W_k)$ for all $k\geq k_0(R)$.
				\item $(M_\infty, F_\infty^*dx^2)$ is a complete Riemannian manifold.
			\end{enumerate}
		In case $M_\infty$ is compact we may take one of the $W_k$ to be $M_\infty$ itself.\end{theorem}
	
	Before proving this theorem, we note that given Lemmas \ref{arralpha} and \ref{balls}, the proof is essentially finished already.  This is because we have shown we can choose a parametrization, at least on Langer charts of a definite positive size, in which each immersion is the graph a function which has small first derivative and bounded higher derivatives; thus Arzela-Ascoli guarantees convergence on each Langer chart. By passing to a subsequence we can add Langer charts so that the convergence agrees on overlaps.  The following merely formalizes this argument. 
	
	\begin{proof}	
	Let $\alpha<\frac{1}{10}$, $r$ given by the lemmas, and $\delta=\frac{r}{10}$.
	
	For each $\ell\in \mathbb{N}$, define
		\begin{equation}
			W_{\ell,k}=\bigcup_{i=1}^{K^\ell} U_{q_k^i,\frac{3\delta}{4}}
		\end{equation}
	where the $q_k^i$ are those points, given by Lemma \ref{balls}, for which $B(p_k,\ell\frac{\delta}{2})\subset \bigcup U_{q_k^i,\frac{\delta}{4}}$.  Let $U_k^i=U_{q_k^i,\delta}$ and $\tilde{U}_k^i=U_{q_k^i,\frac{3\delta}{4}}$.
	
	Fix $\ell$.  For each $1\leq i\leq K^\ell$, $|F_k(q_k^i)|\leq d_k(q_k^i,p_k)\leq \ell\delta$.  Lemma \ref{arralpha} gives Euclidean isometries $A_k^i$ which take $F_k(q_k^i)$ to the origin and $T_{F_k(q_k^i)}F_k(M_k)$ to $\reals^m\times\{0\}$.  Since these Euclidean isometries are bounded, a subsequence of them must converge, for each $i$, to some $A_\infty^i$, which is a Euclidean isometry; moreover since there are finitely many $i$ for a fixed $\ell$, this convergence may be taken to be uniform in $i$.  In particular the $A_k^i$ are a Cauchy sequence.
	
	Lemma \ref{arralpha} also produces $f_k^i:D_\delta^m\rightarrow\reals^n$ so that $\graph f_k^i=A_k^i\circ F_k(U_k^i)$, so that $\left|Df_k^i\right|\leq \alpha$ and $\left|D^2f_k^i\right|\leq \left(1+\alpha^2\right)^\frac{3}{2}C$.  Since the $\{A_k^i\}$ are a Cauchy sequence, for $k,k'$ large enough depending on $\epsilon$, $A_k^i\circ (A_k^i)^{-1}$ is $\epsilon$-close to the identity on $\reals^{m+n}$.  Thus we may take $k,k'$ large enough that $A_{k'}^i\circ (A_k^i)^{-1} (\graph f_k^i|_{D_{\frac{3\delta}{4}}})$ is a graph over $D_\frac{3\delta}{4}$. For any two indices $i,j\leq K^\ell$ with $U_k^i\cap U_k^j$ nonempty, we have that 
		\begin{equation}\begin{aligned}
			A_k^j\circ (A_k^i)^{-1} (\graph f_k^i|_{\pi \circ A_k^i(F_k(U_k^i\cap U_k^j))})=\graph f_k^j|_{\pi \circ A_k^j(F_k(U_k^i\cap U_k^j))}
			\end{aligned}.
		\end{equation}
		  
	We can therefore take $k,k'$ large enough so that 
		\begin{equation}
			A_{k'}^i\circ (A_k^j)^{-1} (\graph f_k^j|_{\pi\circ A_k^j(F_k(\tilde{U}_k^i\cap \tilde{U}_k^j))})
		\end{equation} 
	is a graph over $D_{\frac{3\delta}{4}}\cap (\pi\circ A_{k'}^j(F_{k'}(\tilde{U}_{k'}^i\cap\tilde{U}_{k'}^j)))$.
	
	The previous paragraph allows us to choose $k_0(\ell)$ so that for any $k,k'\geq k_0$, $F_k(W_{\ell,k})$ is a graph over $F_{k'}(W_{\ell,k'})$.  In particular, $W_{\ell,k}$ and $W_{\ell,k'}$ are diffeomorphic and we can write $W_\ell$ unambiguously.  We write $\phi_{\ell,k}$ for the identification of $W_\ell$ with $W_{\ell,k}\subset M_k$.  By construction, $W_\ell\subset W_{\ell+1}$.  In fact, since the $M_k$ are without boundary, for each $W_\ell$ we have $\overline{W}_\ell\subset W_{\ell'}$ for some $\ell'> \ell$.  We therefore can pass to a subsequence so that $\overline{W}_\ell\subset W_{\ell+1}$.
	
	By construction it is clear that $\phi_{\ell+1,k}|_{W_\ell}=\phi_{\ell,k}$, so we can pass to a diagonal subsequence $\phi_k=\phi_{k,k}:W_k\hookrightarrow M_k$.  Setting $M_\infty=\bigcup_{k=1}^\infty W_k$, we have the claimed $M_\infty$ exhausted by the sequence $\{W_k\}$.

	We now prove the convergence of $F_k\circ \phi_k$ on compact sets in $C^{\ell_0+1,\gamma}(M_\infty,\reals^{m+n})$ for any $0\leq\gamma<1$.   Given a compact set $C\subset M_\infty$, $C$ is contained in some $W_K$.  $F_k(\phi_k(W_k))=F_k(W_{k,k})$ is a graph over $F_{k_0}(W_{k_0,k_0})$ for $k\geq k_0(\ell)$; moreover by construction the function of which it is a graph has first derivative bounded by $2\alpha$ and higher derivatives up to order $\ell_0+2$ bounded by Lemma \ref{higherderivs}.  Thus by Arzela-Ascoli, the $F_k\circ\phi_k$ converge in $C^{\ell_0+1,\gamma}(W_\ell,\reals^{m+n})$ for any $0\leq \gamma<1$.  The limit maps $F_\infty:W_\ell\rightarrow \reals^{m+n}$, by construction, agree.  So we have the claimed $F_\infty:M_\infty\rightarrow\reals^{m+n}$. This completes the proof of Theorem \ref{geometric.convergence.euclidean}.
	\end{proof}

	\begin{remark}
	The Cheeger-Gromov charts given by Proposition \ref{CG.charts} are exactly analogous to the Langer charts $U_{r,q}$.  The injectivity bound is used by Cheeger-Gromov to ensure these charts can be taken to be of a definite size; here we are able to exploit, via Lemma \ref{hessianbound}, the bound on $\II$ to achieve this purpose.
	
	In fact we have
		\begin{corollary}[to the proof of Theorem \ref{geometric.convergence.euclidean}]\label{injectivity} The injectivity radius of the induced metric $F^*h$ of an immersed submanifold $F:M\rightarrow(N,h)$ is bounded below:\begin{align*}
				\inj(M,F^*h)\geq \frac{C}{\sup_M\left|\II\right|}
			\end{align*}
			where $C$ depends on the injectivity radius of $(N,h)$.  In case $(N,h)=(\reals^{m+n},dx^2)$, $C=\frac{1}{2\sqrt{2}}$.
		\end{corollary}
		\begin{proof}
			We prove the Euclidean case; the general case is similar.
			Taking $\alpha=1$ and $r$ given by Lemma \ref{arralpha}, we have for any $q\in M$ that $B(q,r)\subset U_{r,q}$ is a graph over the tangent plane at $F(q)$.  Thus $\inj(q)\geq r$.
		\end{proof}
	
	\end{remark}
\subsubsection{General Case}
	The proof of Theorem \ref{geometric.convergence.ambient} proceeds along the same lines as in Theorem \ref{geometric.convergence.euclidean}, using the lemmas in section \ref{section.ambient.construction} in place of those in section \ref{section.euclidean.construction}.  $M_\infty$ is constructed as the union of limits of Langer charts and $N_\infty$ is constructed as the union of limits of Cheeger-Gromov charts.
	
	\begin{remark} We could prove convergence as above given good-enough integral bounds ($L^p, p>m$) on the second fundamental form, as in \cite{langer85}. In the first inequality of (\ref{fundamental.theorem}), we would need to use the Sobolev inequality instead of the fundamental theorem.\end{remark}
	
\section{Application: Topological Finiteness}\label{section.finiteness}
	Before considering applications of Theorem \ref{geometric.convergence.ambient} to the mean curvature flow, which is our main purpose for it, we discuss in this section some topological finiteness theorems which may be of independent interest.

	We begin by relating $C^{\ell,\alpha}$ geometric convergence in the sense of Theorem \ref{geometric.convergence.ambient} to convergence in the function space $C^{\ell,\alpha}(M,N)$.
	
	\begin{proposition}\label{convergences}
		Let $\{M_k\}$ be a family of smooth $m$ manifolds and $\{N_k\}$ a family of smooth $m+n$ manifolds.  If the Riemannian immersions $F_k:(M_k,p_k)\rightarrow (N_k,h_k,x_k)$ converge in $C^{\ell,\alpha}$ in the geometric sense to $F_\infty:(M_\infty,p_\infty)\rightarrow (N_\infty,h_\infty,x_\infty)$, with $M_\infty$ and $N_\infty$ compact, then $\psi_k^{-1}\circ F_k\circ \phi_k$ converge to $F_\infty$ in $C^{\ell,\alpha}(M_\infty,N_\infty)$.
	\end{proposition}
	\begin{proof}
		Since $M_\infty$ and $N_\infty$ are compact, $M_\infty=W_k$ and $N_\infty=V_k$ in the tail of the sequence; thus it makes sense to consider $\psi_k^{-1}\circ F_k\circ \phi_k$ in $C^{\ell,\alpha}(M_\infty,N_\infty)$.  Then Theorem \ref{geometric.convergence.ambient} gives the result.
	\end{proof}

The implicit function theorem gives the following, which says that the set of immersions which are regularly homotopic to a given immersion is open in $C^{1,\alpha}$.

	\begin{proposition}\label{regular.homotopy.open}
		Let $M^m,N^{m+n}$ be smooth compact manifolds, $F\in C^{1,\alpha}(M,N)$ an immersion.  Then there is $\epsilon(F)>0$ such that $||G-F||_{C^{1,\alpha}}\leq \epsilon$ implies that $G$ is an immersion, which is regular-homotopic to $F$ through $C^{1,\alpha}$ immersions.
	
		In particular, the intersection of $C^{1,\gamma}(M,N)$ with each regular-homotopy class is open in $C^{1,\gamma}(M,N)$.
	\end{proposition}
	
	We now apply Theorem \ref{geometric.convergence.ambient} and Propositions \ref{convergences} and \ref{regular.homotopy.open} to obtain a topological finiteness theorem, somewhat analogous to the results in Cheeger's thesis \cite{cheegerthesis}.  We make the following definitions to allow us to state the finiteness theorem.
	
	\begin{definition}
		Two immersions $F,G:M\rightarrow N$ are {\em conjugate regular homotopic} if there exist diffeomorphisms $\phi:M\rightarrow M$ and $\psi:N\rightarrow N$ so that $\psi^{-1} \circ F\circ \phi$ is regular isotopic to $G$. Two embeddings $F,G:M\hookrightarrow N$ are {\em conjugate ambient isotopic} if there exist diffeomorphisms $\phi,\psi$ so that $\psi^{-1}\circ F\circ \phi$ is ambient isotopic to $G$.
	\end{definition}
	
	\begin{theorem}\label{regular.homotopy.finite}
		Consider the class $\mathcal{F}$ of immersions $F:M\rightarrow (N,h)$ which satisfy, for some $C_1,C_2,C_3,C_4,\eta$,\begin{itemize}
				\item $\vol(N,h)\leq C_1$, $\inj(N,h)\geq \eta$, $\left|\Rm(N,h)\right|\leq C_2$
				\item $\vol(M,F^*h)\leq C_3$, $\left|\II(F(M),h)\right|\leq C_4$
			\end{itemize}
		Then there are finitely many diffeomorphism types of $M$, finitely many diffeomorphism types of $N$, and finitely many conjugate regular homotopy classes of $F$ represented in $\mathcal{F}$.
	\end{theorem}
	\begin{proof}
		For any immersion satisfying the hypothesized bounds, $(M,F^*h)$ is a Riemannian manifold with bounded curvature, volume, and by Corollary \ref{injectivity}, injectivity radius. It follows from a standard Riemannian argument that each such $(M,F^*h)$ has bounded diameter.  Cheeger's theorem states that there are finitely many diffeomorphism types of such $M$.\footnote{Alternatively, one can take some care in choosing $\alpha$ to ensure that the Langer charts do not overlap too much, and thus that each contributes a definite volume $V$ depending on $r$ and $\alpha$; then we can ensure the number of Langer charts required to cover $M$ is bounded by $\frac{C_3}{V}$.  The number of diffeomorphism types is then controlled by the combinatorics of how the Langer charts overlap.  This is Langer's approach in \cite{langer85}.}  
		
		So we restrict our attention conjugate regular homotopy classes of immersions from some $M_0$ into some $N_0$. Then the proof of Theorem \ref{geometric.convergence.ambient} allows us to reparametrize $F$ as a $(r,\alpha)$ immersion; in particular, the reparametrized $F$ is bounded in $C^0$ by the diameter bound, bounded in $C^1$ since it is a $(r,\alpha)$-immersion, and bounded in $C^2$ by the assumed bound on the second fundamental form.	
		
		That is, up to reparametrization the class $\mathcal{F}$ is bounded in $C^2(M_0,N_0)$. Hence it is compact in $C^{1,\gamma}(M_0,N_0)$ for any $0\leq\gamma< 1$.  On the other hand, each regular homotopy class is open in $C^{1,\gamma}(M_0,N_0)$.  The theorem follows.
	\end{proof}
	
	By fixing a target manifold, we get a finiteness theorem for regular homotopy classes up to parametrization of the domain:
	
	\begin{theorem}\label{regular.homotopy.finite.fixedtarget}
		For any compact Riemannian manifold $(N,h)$, let $\mathcal{F}_{(N,h)}$ be the class of immersions $F:M\rightarrow(N,h)$ which satisfy $\vol(F(M))\leq C_1$, $\left|\II(F(M))\right|\leq C_2$.  There are finitely many regular homotopy classes, up to parametrization of the domain, represented in $\mathcal{F}_{(N,h)}$.
	\end{theorem}
	
	To state Theorem \ref{regular.homotopy.finite} in a manner more topologically useful, we fix the diffeomorphism type of $M$ and state the contrapositive to obtain:
	
	\begin{theorem}\label{regular.homotopy.finite.fixedtargetdomain2}
		Let $\mathcal{C}=\{c_i\}$ be a collection of regular homotopy classes of maps $F:M\rightarrow (N,h)$, up to diffeomorphism of $M$.  If $\mathcal{C}$ is infinite, then there is no choice of immersed representatives $F_i\in c_i$ which satisfies $\vol(F_i(M))\leq C_1$, $\left|\II(F_i(M))\right|\leq C_2$.
	\end{theorem}
	
	Similarly, we may prove finiteness theorems for ambient isotopy classes of embeddings $F:M\hookrightarrow (N,h)$.  Since embeddedness is fragile, we require uniformity in the following sense:
	
	\begin{definition} The embedding constant of an immersion $F:M\rightarrow (N,h)$ is 
		\begin{equation}
			\kappa(F)=\sup_{p,q\in M} \frac{d_g(p,q)}{d_h(F(p),F(q))}
		\end{equation}
		where $d_g$ is the distance function on $M$ induced by $g=F^*h$ and $d_h$ is the distance function on $N$ induced by $h$.
	\end{definition}
	
	$F$ is an embedding if and only if $\kappa(F)$ is finite.  $F$ is totally geodesic if and only if $\kappa(F)=1$.	
	
	\begin{proposition}\label{ambient.isotopy.open}
		Let $M^m$ be a smooth compact manifold, $F\in C^{1,\alpha}(M,N)$ an embedding.  Then there is $\epsilon(F)>0$ such that $||G-F||_{C^{1,\alpha}}\leq \epsilon$ implies that $G$ is an embedding which is ambient-isotopic to $F$.
		
		In particular, the intersection of $C^{1,\alpha}(M,N)$ with each ambient isotopy class is open in $C^{1,\alpha}(M,N)$.
	\end{proposition}
	\begin{proof}
		The proof is the same as that of Proposition \ref{regular.homotopy.open}, since an immersion which is locally ambient isotopic to an embedding must be an embedding which is ambient isotopic.
	\end{proof}
	
	\begin{theorem}\label{ambient.isotopy.finite}
		Consider the class $\mathcal{F_{\text{emb}}}$ of embeddings $F:M\rightarrow (N,h)$ which satisfy, for some $C_1,C_2,C_3,C_4,C_5,\eta$\begin{itemize}
				\item $\vol(N,h)\leq C_1$, $\inj(N,h)\geq \eta$, $\left|\Rm(N,h)\right|\leq C_2$
				\item $\vol(M,F^*h)\leq C_3$, $\left|\II(F(M),h)\right|\leq C_4$, $\kappa(F)<C_5$
			\end{itemize}
		Then there are finitely many diffeomorphism types of $M$, finitely many diffeomorphism types of $N$, and finitely many conjugate ambient isotopy classes of $F$ represented in $\mathcal{F_{\text{emb}}}$.
	\end{theorem}
	\begin{proof}
		The only difference between the proof of this theorem and Theorem \ref{regular.homotopy.finite} is we must assume the embeddings are uniform so that the class  $\mathcal{F_{\text{emb}}}$ will be closed.
	\end{proof}
	
	Similarly, there are ambient-isotopy versions of Theorems \ref{regular.homotopy.finite.fixedtarget} and \ref{regular.homotopy.finite.fixedtargetdomain}.

	\bigskip

	To conclude this section, we give examples of infinite collections of homotopy classes which have immersive representatives. First consider $M=T^2$, $N=T^5$.  By Whitney's theorem, every map $F:M\rightarrow N$ is homotopic to an immersion.  Moreover, since $T^2$ and $T^5$ are Eilenberg-Maclane spaces, we have $[M,N]=\Hom(\ints^2,\ints^5)$.
	
	Similarly, we may consider two hyperbolic manifolds $M^m=\mathbb{H}^m/\Gamma$, $N^n=\mathbb{H}^{m+n}/\Lambda$, where $\Gamma$ is a lattice in $SO(m,1)$ and $\Lambda$ is a lattice in $SO(m+n,1)$.  If $n\geq m$, Whitney's theorem says that every map from $M$ to $N$ is homotopic to an immersion.  The homotopy classes of maps from $M$ to $N$ are given by $\Hom(\Gamma,\Lambda)$.  $\Gamma$ and $\Lambda$ can be chosen so that $\Hom(\Gamma,\Lambda)$ is infinite.
	
	Or consider the case of a simply-connected four-manifold $X$ with non-torsion $H_2(X)$.  By the theorem of Hurewicz, $|\pi_2(X)|=\infty$; Theorem \ref{regular.homotopy.finite.fixedtargetdomain} says that in order to realize each one of these classes, the immersion must be allowed to have either arbitrarily large volume or arbitrarily large curvature.
	
	We also note that in the case $M=S^1$ and $N$ is closed, every homotopy class admits a geodesic representative, so our finiteness theorems imply that for each $L>0$ there are at most finitely many distinct homotopy classes whose (shortest) geodesic representatives have length less than $L$.

\section{Application: Singularities of the Mean Curvature Flow}\label{section.blowup}
	We now use Theorem \ref{geometric.convergence.ambient} to construct singularity models for compact mean curvature flows $F:M\times[0,T)\rightarrow (N,h)$.

First we state a compactness theorem for mean curvature flows, which follows directly from Theorem \ref{geometric.convergence.ambient}:
\begin{theorem}\label{geometric.convergence.MCF}
	Suppose that $F_j:M_j\times[\alpha,\omega]\rightarrow (N_j,h_j)$ are compact mean curvature flows such that $\left|\II_j(t)\right|\leq C$ for all $j$ and all $t\in[\alpha,\omega]$ and $\left|\nabla^\ell\II_j(0)\right|\leq C_\ell$ for each $\ell$, and such that $(N_j,h_j)$ have uniformly bounded geometry. Then there is a mean curvature flow $F_\infty:M_\infty\times [\alpha,\omega]\rightarrow (N_\infty,h_\infty)$ such that for each $t\in[\alpha,\omega]$, $F_j(t)$ subconverges in $C^{\ell}$ in the geometric sense to $F_\infty(t)$, for any $\ell$; moreover this convergence is uniform in $t$.
\end{theorem}
\begin{proof}
	By the smoothness estimate for the mean curvature flow, the uniform bound on the second fundamental form gives uniform bounds on all its derivatives as well. Thus at each $t\in[\alpha,\omega]$, we may apply Theorem \ref{geometric.convergence.ambient} to get $M_\infty(t)$, $F_\infty(t)$, and $(N_\infty(t),h_\infty(t))$.  The time-derivatives $\frac{\partial^\ell}{\partial t^\ell}F_j$ are, by the flow equation, uniformly bounded; thus $F_\infty(t):M_\infty(t)\rightarrow (N_\infty(t),h_\infty(t))$ are a smooth one-parameter family.  Moreover, the construction of $M_\infty(t)$ and $(N_\infty(t),h_\infty(t))$ and the maps $\phi_j$ and $\psi_j$ relies only on the curvature and injectivity bounds, which are uniform, so we may take $M_\infty$, $(N_\infty,h_\infty)$, and $\phi_j,\psi_j$ independent of time.

	It is clear that $F_\infty:M_\infty\times[\alpha,\omega]\rightarrow (N_\infty,h_\infty)$ is a mean curvature flow.
\end{proof}

\subsection{The Smooth Blow-up}

Now suppose that $F:M\times[0,T)\rightarrow (N,h)$ is a compact mean curvature flow, $(p_j,t_j)$ are a sequence of points and times (the {\em central sequence}), and $\alpha_j\nearrow\infty$ are a sequence of positive numbers such that $\limsup_j\frac{\sup_{M\times[0,t_j]}\left|\II\right|}{\alpha_j}<\infty$. Then the rescales
	\begin{equation}
		\tilde{F}_j(s):F(t_j+\frac{s}{Q_j^2})\rightarrow (N,\alpha_j^2h,F(p_j,t_j))
	\end{equation}
form a sequence as in Theorem \ref{geometric.convergence.MCF} for any $[\alpha,\omega]\subset(-\alpha_j^2t_j,0]$. Note that if the geometry of $(N,h)$ is bounded, then the Cheeger-Gromov limit of $(N,\alpha_j^2h,F(p_j,t_j))$ is $(\reals^{m+n},\dx^2)$.

To construct models for the singularities of the mean curvature flow, we must correctly pick the central sequence $(p_j, t_j)$ and the scale factors $\alpha_j$. The choices we make are inspired by those used by Hamilton for the Ricci flow  \cite{bensbook} \cite{hamiltoncompactness}.
	
	The construction depends on how severe the singularity is.
	
	\begin{proposition}
		For any compact mean curvature flow $F:M\times[0,T)\rightarrow(N,h)$ with singular time $T<\infty$ we have
			\begin{equation}\label{minrate}
				\max_{M}\left|\II(\cdot,t)\right|\geq \frac{C}{\sqrt{T-t}}
			\end{equation}
		where the constant $C$ depends on the initial submanifold $M_0$.
	\end{proposition}

The blow-up rate (\ref{minrate}) is that of a shrinking sphere or cylinder; it represents the mildest sort of singularity that the MCF can encounter.  We define

\begin{definition}
	The mean curvature flow $F:M\times[0,T)\rightarrow (N,h)$ achieves a type I singularity at $T$ if
		\begin{equation*}
			\sup_{M\times[0,T)}\left|\II\right|^2(T-t)<\infty
		\end{equation*}
	Otherwise we say the singularity is of type II.
\end{definition}
	
	First consider the case of a type II singularity. For any sequence $\tilde{t}_j\nearrow T$, let $p_j\in M$  be such that
		\begin{equation}\label{typeIIpicking}
			(\tilde{t}_j-t_j)\left|\II(p_j,t_j)\right|^2=\max_{M\times[0,\tilde{t}_j]}(\tilde{t}_j-t)\left|\II(p,t)\right|^2
		\end{equation}
	Set $Q_j=\left|\II(p_j,t_j)\right|$. By the type II assumption, $(\tilde{t}_j-t_j)Q_j^2\rightarrow\infty$, so for any $t$ there is $j$ large enough that $t\in(-Q_j^2t_j,(\tilde{t}_j-t_j)Q_j^2)$. For such $j$, we compute\begin{equation}\begin{aligned}
		\left|\II_j(p,t)\right|^2=Q_j^{-2}\left|\II(p,t_j+\frac{t}{Q_j^2})\right|^2\ & \leq\  Q_j^2\frac{(\tilde{t}_j-t_j)\left|\II(p_j,t_j)\right|^2}{\tilde{t}_j-(t_j+\frac{t}{Q_j^2})}\\
		&=\frac{(\tilde{t_j}-t_j)Q_j^2}{(\tilde{t_j}-t_j)Q_j^2-t}\end{aligned}\end{equation}
	The right-hand side of this inequality approaches 1 as $j\rightarrow\infty$, hence is bounded by a continuous function of $t$.  Therefore we may apply Theorem \ref{geometric.convergence.MCF} to the $F_j$ to extract a limit mean curvature flow $F_\infty$.
	
	If the singularity is of type I, we pick $t_j=\tilde{t}_j$ and $p_j$ so that $Q_j=\left|\II(p_j,t_j)\right|=\max_{M\times[0,t_j]}\left|\II\right|$. Then $Q_j\rightarrow\infty$.
	
	Since $M$ is compact, in either case we have that, after passing to a subsequence, $p_j\rightarrow\overline{p}$.  We choose the rescales $\tilde{F}_j$ about the central sequence $(\overline{p},t_j)$.
		\begin{equation}\label{rescalesequence}
			\tilde{F}_j(s)=F(t_j+\frac{s}{Q_j^2})\rightarrow (N,Q_j^2h,F(\overline{p},t_j))
		\end{equation}
		
	In the type I case, each $\tilde{F}_j$ has second fundamental form bounded by 1 on the interval $[-Q_j^2t_j,0]$. In the type II case, $\tilde{F}_j$ has second fundamental form bounded by 1 on the interval $[-Q_j^2t_j,Q_j^2(\tilde{t}_j-t_j)]$
		
	\begin{theorem}\label{SBU}
		The geometric limit of the rescaled sequence (\ref{rescalesequence}) is mean curvature flow $F_\infty:M_\infty\times(-\infty,C)\rightarrow\reals^{m+n}$. Here $C=0$ if the singularity is of type I and $C=\infty$ if the singularity is of type II.
		
		Moreover, we have $\left|\II_\infty(p_\infty,0)\right|=1$.
	\end{theorem}
	\begin{proof}
		Note that since $(N,h)$ has bounded geometry, the Cheeger-Gromov limit of $(N,Q_j^2h)$ is $(\reals^{m+n},\dx^2)$.
	
		The only thing left to prove is that $\left|\II_\infty(p_\infty,0)\right|=1$. For a fixed $k$, notice that the rescaled metric $g_k(0)=F_{t_k}^*(Q_k^2h)$ is a metric on $M$. Let $B_k$ denote the metric ball in the metric $g_k(0)$. Since $p_j\rightarrow p$, we have that for any $R>0$, $p_j\in B_k(\overline{p},R)$ for all $j\geq j_0(k,R)$. By geometric convergence the metrics $\{g_j(0)\}$ have the Cauchy property that $B_k(\overline{p},R)\subset B_j(\overline{p},2R)$.
		
		On the other hand, $p_j$ is a point where $\left|\II_j(p_j,0)\right|=1$. Thus in the tail of the sequence there is a point of curvature 1 within $2R$ of $\overline{p}$. This condition clearly persists to the limit, so there is a point of curvature 1 within $2R$ of $p_\infty$. But $R$ was arbitrary, so letting $R\rightarrow 0$ we see that $\left|\II_\infty(p_\infty,0)\right|=1$.
	\end{proof}
	
	We refer to the MCF $F_\infty:M_\infty\times(-\infty,C)\rightarrow \reals^{m+n}$ as a {\em smooth blow-up} of the original flow $F:M\times[0,T)\rightarrow (N,h)$.
	
	Though we have stated the construction of the smooth blow-up for compact mean curvature flows, note that the construction will also work provided the singularity is of {\em compact type}:
		\begin{definition}
		  We say that a mean curvature flow $F:M\times[0,T)\rightarrow (N,h)$ has a compact-type singularity at $T<\infty$ if:\begin{itemize}
		  	\item $\lim_{t\rightarrow T}\sup_M\left|\II(t)\right|=\infty$
		  	\item For any $t_j\nearrow T$, there exist $p_j$ with $\left|\II(p_j,t_j)\right|=\sup_{M\times[0,t_j)}\left|\II\right|$ and $p_j\rightarrow p$
		    \end{itemize}
		\end{definition}

\begin{remark}
	In general smooth blow-ups are nonunique, since Theorem \ref{geometric.convergence.MCF} only gives subsequential convergence.
\end{remark}
	
\begin{remark}
	The diffeomorphisms $\phi_k$ in the construction of the smooth blow-up amount to choosing the ``correct" parametrization of regions of the domain submanifold $M$ which are becoming singular. Huisken-Sinestrari, in order to carry out their surgery theorem, explicitly construct such a parametrization of the singular region by means of a nearby shrinking cylinder \cite{huisken&sinestrari08}. The import of Theorem \ref{geometric.convergence.MCF} is that such a parametrization can always be found.
\end{remark}

\subsection{Comparison to the Tangent Flow}
The smooth blow-up is inspired Hamilton's idea for singularity models for the Ricci flow \cite{bensbook}. In previous literature on the mean curvature flow, singularities have been understood using a rescaling procedured called the {\em tangent flow}, which we now describe.

To produce a tangent flow, we work in the category of {\em Brakke flows}, i.e. one-parameter families of integral currents which are locally maximally area-decreasing \cite{brakke78} \cite{ilmanen94}. A mean curvature flow is, a fortiori, a Brakke flow. We have the following theorem due to Brakke, which follows from the compactness theorem for integral currents of Federer-Fleming.

	\begin{theorem}[Brakke, \cite{brakke78}]
		Let $T_k(t)$ be a sequence of Brakke flows on $[\alpha,\omega]$. Then $T_k(t)$ subconverge as integral currents to a Brakke flow $T_\infty(t)$ on $[\alpha,\omega]$.
	\end{theorem}

Given a compact mean curvature flow $F_t:M\times[0,T)\rightarrow\reals^{m+n}$, there is some point $x_0\in \reals^{m+n}$ such that $\lim_{t\rightarrow T} F(p,t)=x_0$ for some $p\in M$ with $\lim_{t\rightarrow T}\left|\II(p,t)\right|=\infty$. We say that the singularity of the flow occurs at $x_0$. If $t_j\nearrow T$ and $Q_j=\sup_{M\times[0,t_j]}\left|\II\right|$, we define
	\begin{equation}
		\overline{F}_j(p,s)=Q_j^2\left[F\left(p,T+\frac{s}{Q_j^2}\right)-x_0\right]
	\end{equation}
and call a subsequential Brakke flow limit of $\overline{F}_j$ a {\em tangent flow} with center $(x_0,T)$ of the original flow $F_t$.

	The primary advantage of using the tangent flow construction is that all Brakke flows which arise as tangent flows satisfy an elliptic equation called the {\em self-shrinker} equation.
	
\begin{definition}
	Given a mean curvature flow $M(t)$ and any $(x_0,t_0)\in\reals^ {m+n}\times \reals$, we define Huisken's monotonic quantity
		\begin{equation}
			\Theta_{M,x_0,t_0}(t)=\int_{M(t)}(4\pi (t_0-t))^{-\frac{m}{2}}e^{\frac{-|x-x_0|^2}{4(t_0-t)}}d\Haus^m
		\end{equation}
\end{definition}

\begin{theorem}[Huisken \cite{huisken90}]\label{huisken.monotonicity}
	Huisken's monotonic quantity is monotone along a smooth mean curvature flow. In particular it satisfies:
		\begin{equation*}
			\frac{d}{dt}\Theta_{M,x_0,t_0}(t)=-\int_{M(t)}\left|H+\frac{1}{2(t_0-t)}(x-x_0)^\perp\right|^2(4\pi (t_0-t))^{-\frac{m}{2}}e^{\frac{-|x-x_0|^2}{4(t_0-t)}}d\Haus^m
		\end{equation*}
	here $(x-x_0)^\perp$ is the projection of the vector $x-x_0$ to the normal bundle of $M$.
\end{theorem}

Flows for which $\Theta$ is constant are called {\em self-shrinking}. In fact the mean curvature flow with intial data satisfying the elliptic equation
	\begin{equation}\label{self.shrinker.eqn}
		H=\alpha x^\perp
	\end{equation}
	for some $\alpha<0$ are necessarily self-shrinking; we call a submanifold, or more generally an integral current, satisfying (\ref{self.shrinker.eqn}) a {\em self-shrinker}.  We have the following theorem:
	
\begin{theorem}[Huisken \cite{huisken90}]\label{tangent.flow.SS}
	Any tangent flow to a mean curvature flow is a self-shrinking flow.
\end{theorem}

The self-shrinking condition imposes fairly strong restrictions, as in the following theorem:

\begin{theorem}[Huisken \cite{huisken93}]
	A smooth mean-convex self-shrinking hypersurface must be one of the following:\begin{itemize}
		\item a round sphere
		\item a round cylinder
		\item $\Gamma\times\reals^{m-1}$, where $\Gamma$ is one of the Abresch-Langer curves \cite{abresch&langer86}
	\end{itemize}
\end{theorem}

Huisken \cite{huisken90} showed that in the type I case, the tangent flow construction in fact yields a smooth limit. We now show that this construction is the same as the smooth blow-up.
\begin{proposition}\label{SBU.SS}
	Suppose that $F_t:M\times[0,T)\rightarrow\reals^{m+n}$ is a compact mean curvature flow with type I singularity at $T$. Then the smooth blow-up of $F_t$ is a self-shrinking flow.
\end{proposition}
\begin{proof}
	The proof is the same as the proof of Theorem \ref{tangent.flow.SS}, with the necessary changes enabled by the type I assumption.
	
	Given the central sequence $\{(\overline{p},t_j)\}$, set $x_j=F(\overline{p},t_j)$. Then there is a subsequential limit $x_0=\lim_j x_j$. We compute:\begin{equation}\begin{aligned}
				|x_j-x_0|=&\left|\int_{t_j}^TH(\overline{p},s)ds\right|\\
				\ \leq& \ \int_{t_j}^T|H(\overline{p},s)|ds\\
				\ \leq& \ \int_{t_j}^TC(T-s)^{-\frac{1}{2}}ds\\
				=&C(T-t_j)^\frac{1}{2}\leq\frac{C'}{Q_j}\end{aligned}
			\end{equation}
		Thus $\{Q_j(x_0-x_j)\}$ is a bounded sequence, so that again passing to a subsequence, we have some $\overline{x}=\lim_jQ_j(x_0-x_j)$.
		
		Set $\alpha_j=Q_j^2(T-t_j)$.  Then each $M_j$ exists on $(-Q_j^2t_j,\alpha_j)$.  By the type I assumption, we can pass to a subsequence so that the limit $\lim_j\alpha_j=C$ exists. We consider Huisken's monotonic quantity centered at $(\overline{x},C)$:
			\begin{equation}
				\Theta_{M_\infty,\overline{x},C}(s)=\int_{M_\infty(s)}(4\pi (C-s))^{-\frac{m}{2}}e^{-\frac{|x-\overline{x}|^2}{4(C-s)}}d\Haus^m
			\end{equation}
	
	Given any $-Q_j^2t_j<a<b<\alpha_j$ and a compact set $K\subset\reals^{m+n}$, we have by the scaling properties of Huisken's quantity:
		\begin{multline}\label{integrate.dt}
			\int_a^b\int_{M_j(s)\cap K}\left|H+\frac{(x-Q_jx_0)^\perp}{2(\alpha_j-s)}\right|^2(4\pi (\alpha_j-s))^{-\frac{m}{2}}e^{-\frac{|x-Q_jx_0|^2}{4(\alpha_j-s)}}d\Haus^mds\\
			=\int_{t_j+\frac{a}{Q_j^2}}^{t_j+\frac{b}{Q_j^2}}\int_{M(t)\cap (Q_j^{-1}K+x_j)}\left|H+\frac{(x-x_0)^\perp}{2(T-t)}\right|^2(4\pi (T-t))^{-\frac{m}{2}}e^{-\frac{|x-x_0|^2}{4(T-t)}}d\Haus^m dt
			\end{multline}
			
	We can estimate the right-hand side of (\ref{integrate.dt}) by integrating over all of $M(t)$ and applying Theorem \ref{huisken.monotonicity}:
	\begin{multline}\label{monotonicity.comp}
		\int_{t_j+\frac{a}{Q_j^2}}^{t_j+\frac{b}{Q_j^2}}\int_{M(t)\cap (Q_j^{-1}K+x_j)}\left|H+\frac{(x-x_0)^\perp}{2(T-t)}\right|^2(4\pi (T-t))^{-\frac{m}{2}}e^{-\frac{|x-x_0|^2}{4(T-t)}}d\Haus^m dt\\
		\begin{aligned}\ & \leq\ \int_{t_j+\frac{a}{Q_j^2}}^{t_j+\frac{b}{Q_j^2}}\int_{M(t)}\left|H+\frac{(x-x_0)^\perp}{2(T-t)}\right|^2(4\pi (T-t))^{-\frac{m}{2}}e^{-\frac{|x-x_0|^2}{4(T-t)}}d\Haus^m dt\\
		&=\Theta_{M,x_0,T}(t_j+\frac{a}{Q_j^2})-\Theta_{M,x_0,T}(t_j+\frac{b}{Q_j^2})\end{aligned}\end{multline}
	Since $t_j+\frac{a}{Q_j^2}$ and $t_j+\frac{b}{Q_j^2}$ both approach $T$ as $j\rightarrow\infty$, we have by Theorem \ref{huisken.monotonicity} that the right-hand side of (\ref{monotonicity.comp}) goes to 0 as $j\rightarrow\infty$.
	
	On the other hand, the left-hand side of (\ref{integrate.dt}) approaches
	\begin{equation}
		\int_a^b\int_{M_\infty(s)\cap K}\left|H+\frac{(x-\overline{x})^\perp}{2(C-s)}\right|^2(4\pi (C-s))^{-\frac{m}{2}}e^{-\frac{|x-\overline{x}|^2}{4(C-s)}}d\Haus^mds
	\end{equation}

	Since $K$, $a$, and $b$ were arbitrary we have that for almost every $s$ and almost every $x\in M_\infty(s)$ that $\left|H+\frac{(x-\overline{x})^\perp}{2(C-s)}\right|^2=0$. Thus $M_\infty$ is a self-shrinking flow with center $(\overline{x},C)$.
\end{proof}

We therefore have the following characterization of singularity types in case $(N,h)=(\reals^{m+n},\dx^2)$.

\begin{corollary}
	The singularity of a compact mean curvature flow $F_t:M\rightarrow\reals^{m+n}$ is of type I if and only if it admits a smooth blow-up which becomes extinct in finite time.
\end{corollary}

\section*{Acknowledgements}
	The author wishes to thank his adviser Jon Wolfson for his help and suggestions.  He also wishes to thank Thomas Parker and Natasa Sesum for several helpful discussions.

\bibliography{mcf}{}
\bibliographystyle{plain}

\end{document}